\documentclass[orivec]{llncs}
\usepackage{amsmath,amssymb,amsfonts}
\usepackage{theorem}
\usepackage{graphicx}
\usepackage{times}
\usepackage{color}
\usepackage{mathrsfs}
\usepackage{upgreek}
\usepackage{amsmath}
\usepackage{todonotes}

\newtheorem{Def}{Definition}
\newtheorem{Theo}{Theorem}
\newtheorem{cor}{Corollary}
\newtheorem{rem}{Remark}
\newtheorem{pro}{Proposition}
\newtheorem{examp}{Example}
\newtheorem{Lemma}{Lemma}
\newcommand{\rea}{\mathbb{R}}

\newcommand{\nat}{\mathbb{N}}
\newcommand{\x}{\textbf{\textit{x}}}
\newcommand{\y}{\textbf{\textit{y}}}
\newcommand{\z}{\textbf{\textit{z}}}

\begin{document}

\title {Decomposition approaches to integration}
\author{
Salvatore Greco\thanks{\email{salgreco@unict.it}}\inst{1}\inst{2}
\and Radko Mesiar\thanks{\email{radko.mesiar@stuba.sk}} \inst{3}\inst{4}
\and Fabio Rindone\thanks{\email{frindone@unict.it}} \inst{1}
\and Ladislav Sipeky\thanks{\email{ladislav.sipeky@stuba.sk}}\inst{3}}
\institute{Department of Economics and Business\\
95029 Catania, Italy\\
\and
Portsmouth Business School, Operations \& Systems Management University of Portsmouth, Portsmouth PO1 E, United Kingdom\\
\and
Department of Mathematics and Descriptive Geometry, Faculty of Civil Engineering\\
Slovak University of Technology\\
 Bratislava, Slovakia\\
\and
Institute of Theory of Information and Automation\\
Czech Academy of Sciences\\
Prague, Czech Republic\\
}

\maketitle
\begin{abstract}
Following the idea of Even and Lehrer \cite{EvenLehrer2014}, we discuss a general approach to integration based on decomposition of the integrated function. 
We distinguish sub-decomposition based integrals (in economics linked with optimization problems to maximize the possible profit) and super-decomposition based integrals (linked with costs minimization).
We provide several examples (both theoretical and realistic) to stress that our approach generalizes that of Even and Lehrer \cite{EvenLehrer2014} and also covers problems of linear programming and combinatorial optimization.
Finally, we introduce some new types of integrals related to optimization tasks. 
\end{abstract}

\section{Introduction}
%\todo[size=\footnotesize,color=green!40]{this condition implies $\mathcal{D}\neq \emptyset$ and, moreover, avoids the trivial case $\mathcal{D}=\{(\boldsymbol 0)\}$, for which no weighting function is admitted.}
The idea of decomposition of the integrated function $f$ for the integration purposes is a basic feature of constructions / definitions of integrals since ever. Recall, e.g., Eudoxus of Cnidus (408–-355 BC) exhaustion principle, Riemann and Lebesgue integrals (lower and upper integral sums), etc.. Integration always merges two sources of information, the integrated function and weights of special functions used for decomposition purposes (e.g., measures assigning weights to sets, i.e., to characterize functions of sets), into a single representative value. In this contribution, we will deal with non-negative (measurable) functions and non-negative weights only, supposing always the monotonicity of the considered weights, and vanishing of such weights for  null functions. Both from transparency of our ideas as well as for the application purposes in economics and multicriteria decision support, we will always deal with a fixed finite space $N=\{1,\ldots,n\}$, where $n\in\mathbb{N}$ is a fixed positive integer. Then the power set $2^N$ being considered excludes any measurability constraints. Each function $f:N\to[0,\infty[=\mathbb{R}_+$ can be seen as an n-dimensional vector $\x\in\mathbb{R}_+^n$, $\x=(x_1,\ldots,x_n)=(f(1),\ldots,f(n))$. The aim of this contribution is a proposal of a general approach to decomposition based integration, distinguishing sub-decompositions and super-decompositions. We will stress several integrals known from the literature as particular instances of our approach. Moreover, several new types of integrals related to optimization tasks will be introduced and exemplified. 
The paper is organized as follows.
In Section \ref{subdec} we propose the idea of sub-decomposition based integrals and, similarly, super-decomposition approach to integration is discussed in Section \ref{superdec}. We provide several examples of application of decomposition integrals, both theoretical as well as realistic.
In Section \ref{relation} we confront our approach with previous research in literature, especially with the idea of Even and Lehrer \cite{EvenLehrer2014}. 
Particular decomposition based integrals are discussed in Section \ref{particular}. 
Finally, some concluding remarks and formal proposal for future researches are added in Section \ref{concluding}. 

\section{Sub-decomposition based integrals}\label{subdec}

Any system of vectors of $\mathbb{R}_+^n$, $(\x^i)_{i=1}^k=(\x^1,\ldots,\x^k)\in(\rea^n_+)^k$ with $k\in\nat$, is called a \textit{collection}, and the set of all collections is $\mathcal{R}_n=\cup_{k\in\nat}(\rea^n_+)^k$. 
A \textit{decomposition system} is any $\mathcal{D}\subseteq\mathcal{R}_n$ 
such that there exists $\x\neq \boldsymbol 0=(0,\ldots,0)$ with $\x\in (\x^i)_{i=1}^k$ for some collection  $(\x^i)_{i=1}^k\in \mathcal{D}$.

Given a decomposition system $\mathcal{D}$, we denote
$$\tilde{\mathcal{D}}=\left\{\x \in\mathbb{R}_+^n|\ \x \in(\x^i)_{i=1}^k\ \mbox{ for some collection } (\x^i)_{i=1}^n\in\mathcal{D}\right\}.$$

Conversely, for any $X\subseteq\rea^n_+$, we define 
$$\mathcal D_X=\{(\x^i)_{i=1}^k\in\mathcal{R}_n\ |\ \x^i\in X\ \mbox{for all} \ i=1,\ldots,k\}$$ 
as the complete decomposition system generated by $X$, and clearly $\tilde{\mathcal D}_X=X$ and, moreover, $\mathcal D_X$ is the union of all decomposition systems $\mathcal D$ such that $\tilde{\mathcal D}=X$.

\begin{Def}\label{weighting}

Let $\mathcal{D}$ be a decomposition system. A mapping $A:\tilde{\mathcal{D}}\to\mathbb{R}_+$ is called a weighting function on $\mathcal{D}$ whenever
\begin{itemize}
    \item $A(\x)\leq A(\y)$ if $\x,\y\in\tilde{\mathcal{D}}, \x\leq\y$ (monotonicity),
    \item $A(\x)>0$ for some $\x\in\tilde{\mathcal{D}}$ and $A(\mathbf{0})=0$ whenever $\mathbf{0}\in\tilde{\mathcal{D}}$ (boundary conditions).

\end{itemize}
\end{Def}

Observe that if $\tilde{\mathcal{D}}=\mathbb{R}_+^n$, then any weighting function A can be seen as an aggregation function (in the sense of \cite{grabisch2009aggregation}, with related boundary condition, i.e., $\sup \{A(\x)\ |\ \x\in\mathbb{R}_+^n\}=+\infty$ replaced by $\sup \{A(\x)\ |\ \x\in\mathbb{R}_+^n\}>0$).

The following example is inspired to Even and Lehrer \cite[example in Section 2]{EvenLehrer2014}.

\begin{examp}\label{workers}

Consider two different work agencies $\mathcal A_1$ and $\mathcal A_2$.
Each agency provides a couple of workers with exactly the same skills. 
However, each of the four workers can work alone, or together with one or more partners.
The possible teams are identified with  
\[
\mathcal T=\left\{0,1,2\right\}^2\setminus\left\{\left(0,0\right)\right\}\subseteq \nat_0^2,
\]
where $(1,0)$, $(0,1)$ represent basic teams formed by a single worker from agency $\mathcal A_1$ and $\mathcal A_2$ respectively, while, e.g., $(2,1)$ is the team formed by the two workers from $\mathcal A_1$ and one indifferently chosen from $\mathcal A_2$. 
Suppose we know the efficiency of each team, measured in some work unit, given by the weighting function 
$E:\mathcal T\rightarrow \rea_+$:
\begin{equation}
\left\{
\begin{array}{l}
E(1,0)=1.0\\
E(2,0)=2.2\\
E(0,1)=1.1\\
E(0,2)=2.0\\
E(1,1)=2.2\\
E(2,1)=3.5\\
E(1,2)=3.0\\
E(2,2)=4.3\\
\end{array} \right.
\label{eq:E}
\end{equation}
Clearly, we want to maximize our efficiency by choosing the best group of teams within the decomposition system 
(let us note that $\tilde{\mathcal D}_\mathcal T=\mathcal T$)
\[
\mathcal D_\mathcal T=\left\{\left(\x^j\right)_{j=1}^k\ |\ \x^j\in\mathcal T\ \mbox{with}\ \sum_{j=1}^k{\x^j}\le (2,2)\right\}.
\]
We will return on this example later.

\end{examp}

Let $\mathcal{D}\subseteq \mathcal{R}_n$ be a decomposition system and let $A:\tilde{\mathcal{D}}\to\mathbb{R}_+$ be a weighting function on $\mathcal{D}$. From now, we call $(A,\mathcal{D})$ a \textit{base for integration} on $\rea^n_+$ (shortly, a \textit{base}).
Given a base $(A,\mathcal{D})$, a vector $\x\in\rea^n_+$ is called $(A,\mathcal{D})$-sub-integrable if (we use the convention $\sup\emptyset=0$)

\begin{equation}\label{subdE}
\sup{\{\sum_{j=1}^k A(\y^j)\ |\ (\y^j)_{j=1}^k \in \mathcal{D},\ \sum_{j=1}^k\y^j\leq \x \}}<+\infty,
\end{equation}
and we define the set of $(A,\mathcal{D})$-sub-integrable vectors as

\[
S_{\left(A,\mathcal D\right)}=\left\{\x\in\rea^n_+\ | \ \x\mbox{ is $(A,\mathcal{D})$-sub-integrable }\right\}.
\]
Let us note that $S_{(A,\mathcal D)}\neq\emptyset$, since the null vector $\boldsymbol 0$ is $(A,\mathcal{D})$-sub-integrable for any base. 

Now, we can introduce our sub-decomposition based integral. 

\begin{Def}\label{subdecompositionI}
Let $(A,\mathcal{D})$ be a base for integration on $\rea^n_+$, then the $(A,\mathcal{D})$-based sub-decomposition integral is the functional $I_{(A,\mathcal{D})}:S_{\left(A,\mathcal D\right)}\to\mathbb{R}_+$ defined by 
\begin{equation}\label{subdE1}
I_{\left(A,\mathcal{D}\right)}(\x)=\sup{\{\sum_{j=1}^k A(\y^j)\ |\ (\y^j)_{j=1}^k \in \mathcal{D},\ \sum_{j=1}^k\y^j\leq \x \}}.
\end{equation}
\end{Def}

The following Lemma \ref{mon} follows directly by definitions of $S_{\left(A,\mathcal D\right)}$ and
$I_{\left(A,\mathcal D\right)}$.

\begin{lemma}\label{mon}
For all $\y\in S_{\left(A,\mathcal D\right)}$ and $\x\leq\y$, then $\x\in S_{\left(A,\mathcal D\right)}$ and $I_{\left(A,\mathcal D\right)}(\x)\le I_{\left(A,\mathcal D\right)}(\y)$.
\end{lemma}

\begin{rem}\label{super}
Let $\mathcal D$ be a decomposition system and let $B:\rea^n_+\rightarrow \rea_+$ be an aggregation function which is  super-additive $[B(\x+\y)\ge B(\x)+B(\y)]$,
then for the weighting function $A=B|_{\tilde{\mathcal{D}}}$ it holds $S_{\left(A,\mathcal D\right)}=\rea^n_+$ and 
$I_{(A,\mathcal{D})}(\x)\le B(\x)$ for each $\x\in\rea^n_+$.
Obviously, if $(\x)\in\mathcal{D}$, then $I_{(A,\mathcal{D})}(\x)= B(\x)$.
%and if there exists $(\y^j)_{j=1}^k\in \mathcal{D}$ such that $\sum_{j=1}^k \y^j =\x$, then $I^{(A,\mathcal{D})}(\x)= A(\x)$.
\end{rem}

Regarding the domain $S_{\left(A,\mathcal D\right)}$ of sub-decomposition integral $I_{\left(A,\mathcal D\right)}$, this depends on both $A$ and $\mathcal{D}$. Suppose that $(A,\mathcal{D})$ and $(A',\mathcal{D'})$ are two bases such that $\mathcal{D}\subseteq \mathcal{D'}$ and $A\le A'|_{\tilde{\mathcal{D}}}$, then $S_{\left(A,\mathcal D\right)}\supseteq S_{\left(A',\mathcal{D'}\right)}$ and 
$I_{\left(A,\mathcal{D}\right)}\le I_{\left(A',\mathcal{D'}\right)}$ on $S_{\left(A',\mathcal{D'}\right)}$.
This will be clear also in the following relevant examples.

\begin{examp}
Consider, e.g., $n=2$, $\mathcal{D}=\mathcal{R}_2$ and $A:\mathbb{R}_+^2\to \mathbb{R}_+$ given by $A(x,y)=x$. 
Then $S_{\left(A,\mathcal{R}_2\right)}=\rea^n_+$ and $I_{\left(A,\mathcal{R}_2\right)}(x,y)=x$.
If we consider the weighting function $A'(x,y)=x+\sqrt{y}$, then $S_{\left(A',\mathcal{R}_2\right)}=\left\{(x,0)\ |\ x\in\mathbb{R}_+\right\}$ and $I_{\left(A',\mathcal{R}_2\right)}(x,0)=x$, while for any $x\in\mathbb{R}_+$ and $y>0$, $(x,y)$ is non $(A',\mathcal{R}_2)$-sub-integrable, indeed, being $A'$ sub-additive, 
$\sup{\{\sum_{i=1}^k A'(x_i,y_i)\ |\ \sum_{i=1}^k (x_i,y_i)=(x,y)\}}\ge$ $ \lim_{n\rightarrow +\infty}nA'(\frac{x}{n},\frac{y}{n})=\lim_{n\rightarrow+\infty}(x+\sqrt{n}y)=+\infty$. 
\end{examp}

Consider a set of objects (criteria) $N=\{1,\ldots,n\}$, and define a \textit{chain} a system $\left(E_j\right)_{j=1}^k$ such that $ E_1 \subseteq \ldots \subseteq E_k\subseteq N$. Let $\mathcal{D}$ be the set of all collections $\left(c_j\cdot\mathbf{1}_{E_j}\right)_{j=1}^k$, being $c_j$ positive constants and $\left(E_j\right)_{j=1}^k$ a chain.
Now consider the weighting functions $A:\tilde{\mathcal{D}}\to \mathbb{R}_+$, determined by a monotone measure $m:2^N\to\mathbb{R}_+$ ($m(\emptyset)=0$, $m(N)>0$, and $m(E_1)\leq m(E_2)$ whenever $E_1 \subseteq E_2\subseteq N $), by means of $A(c\cdot\mathbf{1}_E)=c\cdot m(E)$.
In this case $S_{\left(A,\mathcal D\right)}=\rea^n_+$ and $I_{\left(A,\mathcal D\right)}$ is the Choquet integral \cite{choquet1953theory} with respect to measure $m$.

Other that for the Choquet integral, in majority of integrals known so far (Lebesque, Choquet, Shilkret, Concave, Pan, etc. integrals), decomposition systems $\mathcal{D}$  such that any $\x\in\tilde{\mathcal{D}}$ can be written in the form $c\cdot\mathbf{1}_E$, where $c$ is a positive constant and $E$ a subset of $N$ ($\mathbf{1}_E$ is the corresponding characteristic function) are considered, and the corresponding weighting functions $A:\tilde{\mathcal{D}}\to \mathbb{R}_+$ are then determined by $A(\x)=A(c\cdot\mathbf{1}_E)=c\cdot m(E)$, being $m:2^N\to\mathbb{R}_+$ monotone measures. 
Hence all these integrals are covered by our approach. For more details see Sections \ref{relation} and \ref{particular}.

Typical economical problems deal with finite number of goods $g_1,\ldots,g_n$, and then weight (price) is assigned to groups of goods represented by multisets, i.e., vectors $\x\in\mathbb{N}_0^n$ where $\mathbb{N}_0=\{0,1,2,\ldots\}$. Note that due to limitations in storing/production, $\tilde{\mathcal{D}}$ is then mostly finite.
For this purpose, the next result is important for real applications.

\begin{pro}\label{base3}
For any base $(A,\mathcal{D})$ such that $\tilde{\mathcal{D}}$ is finite, then $S_{\left(A,\mathcal D\right)}=\rea^n_+$.
\end{pro}

\begin{proof}
Let $\tilde{\mathcal{D}}=\{\x^1,\ldots,\x^m\}$ be finite, and without loss of generality, we can imagine that $\x^j\ne\boldsymbol 0$, $j=1,\ldots,m$. Now, for any $\y\in\rea^n_+$ there exist $n_1,\ldots,n_m\in\nat$ such that for each $j=1,\ldots,m$ the vector $n_j\x^j$ exceeds $\y$ in some component.
Thus, we have
\[
\sup{\{\sum_{j=1}^k A(\y^j)\ |\ (\y^j)_{j=1}^k \in \mathcal{D},\ \sum_{j=1}^k\y^j\leq \y    \}}\leq n_1A(\x^1)+\ldots + n_kA(\x^m)<+\infty.
\]

\begin{flushright}
$\square$
\end{flushright}

\end{proof}

\begin{rem}
%\todo[size=\footnotesize,color=green!40]{must we provide examples of what claimed in this remark? I have (Fabio)}
If $(A,\mathcal{D})$ is a base such that $S_{\left(A,\mathcal D\right)}=\rea^n_+$, then 
$I_{\left(A,\mathcal D\right)}:\rea^n_+\rightarrow \rea_+$ is a weighting function on $\rea^n_+$. 
Indeed, monotonicity of $I_{\left(A,\mathcal D\right)}$ and condition $I_{\left(A,\mathcal D\right)}(\boldsymbol 0)=0$ follow by definition and, moreover, since $A$ is a weighting function, there exists $\x\in\tilde{\mathcal{D}}$ such that $A(\x)>0$; suppose that $\x$ belongs to the collection $(\y^j)_{j=1}^k \in \mathcal{D}$, it follows that $I_{\left(A,\mathcal D\right)}\left(\sum_{j=1}^k\y^j\right)\ge \sum_{j=1}^k A(\y^j) \ge A(\x)>0$.\\
Let us note that $I_{(A,\mathcal{D})}$ restricted on $\tilde{\mathcal{D}}$ is not, in general, a weighting function, consider, e.g., 
$\mathcal{D}\left\{\left(\left(1,3,0\right),\left(3,1,0\right)\right)\right\}$ with $A(1,3,0)=A(3,1,0)=2$, then 
$I_{\left(A,\mathcal D\right)}(1,3,0)=I_{\left(A,\mathcal D\right)}(3,1,0)=\sup\emptyset=0$. 
Now suppose there exists $\x\in \tilde{\mathcal{D}}$ such that $I_{\left(A,\mathcal D\right)}(\x)>0$ and then $I_{\left(A,\mathcal D\right)}$ is a weighting function on $\tilde{\mathcal{D}}$ and we can consider $I_{\left(I_{\left(A,D\right)},D\right)}$. However, in this case, the two weighing function $A$ and $I_{\left(A,\mathcal D\right)}$ are non comparable and also considering sub-decomposition integrals, we have that $I_{\left(A,\mathcal D\right)}$ is non comparable with $I_{\left(I_{\left(A,D\right)},D\right)}$, see Example \ref{comp}. %[neither with  $I_{\left(I_{\left(A,D\right)},\mathcal{R}_n\right)}$].
\\
A case where $I_{\left(A,\mathcal D\right)}$ and $I_{\left(I_{\left(A,D\right)},D\right)}$ are comparable is when the weighting function $A$ is super-additive, since in this case for any $\x$ and any collection $(\y^j)_{J=1}^k\in\mathcal{D} $ such that $\sum_{j=1}^k\y^j\le\x$, it follows that $\sum_{j=1}^kA(\y^j)\le A(\sum_{j=1}^k\y^j)\le A(\x)$, and then 
$I_{\left(A,\mathcal D\right)}(\x)\le A(\x)$ and, consequently, $I_{\left(I_{\left(A,D\right)},D\right)}\le I_{\left(A,\mathcal D\right)}$.
\\
Finally, let us note that when $\mathcal D=\mathcal{R}_n$ and $S_{(A,D)} =\rea_+^n$,  then $I_{\left(A,\mathcal D\right)}= I_{\left(I_{\left(A,D\right)},D\right)}$.
\end{rem}

\begin{examp}\label{comp}
Consider  
$\mathcal{D}\left\{\left(\left(0,2,1\right),\left(2,0,0\right)\right),\left(\left(2,2,1\right),\left(0,1,2\right)\right),\left(\left(0,1,2\right)\right)\right\}$
and the weighting function  $A(0,2,1)=A(2,0,0)=A(0,1,2)=2$, $A(2,2,1)=3$.\\
It follows that $I_{\left(A,\mathcal D\right)}(0,2,1)=I_{\left(A,\mathcal D\right)}(2,0,0)=\sup\emptyset=0$, $I_{\left(A,\mathcal D\right)}(0,1,2)=2$ and $I_{\left(A,\mathcal D\right)}(2,2,1)=4$, $I_{\left(A,\mathcal D\right)}(2,3,3)=5$.\\
Finally, it is easily computed that 
$I_{\left(I_{\left(A,D\right)},D\right)}(2,2,1)=0$ and 
$I_{\left(I_{\left(A,D\right)},D\right)}(2,3,3)=6$. 
\end{examp}

When $\mathcal{D}=\mathcal{R}_n$, we are able to enunciate sufficient conditions for existence of 
$I_{\left(A,\mathcal{R}_n\right)}$ on all $\rea^n_+$ (for the proof of Theorem \ref{suff} and subsequent corollaries, see \cite{GMRS}).

\begin{Theo}\label{suff}
$S_{\left(A,\mathcal{R}_n\right)}=\rea^n_+$ if and only if the constant vector $\mathbf{1}=(1,\ldots,1)$ is $(A,\mathcal{D})$-sub-integrable.
\end{Theo}

\begin{cor}\label{base1}
Let $A:\rea^n_+\to\rea_+$ be a weighting function on $\mathcal{R}_n$ such that for each $\y\in\rea^n_+$, 
$A(\x)\leq c\cdot\max{\{y_1,\ldots,y_n\}}$, where $c$ is some fixed constant from $\left]0,\infty\right[$. Then $S_{\left(A,\mathcal{R}_n\right)}=\rea^n_+$.
\end{cor}

Due to Corollary \ref{base1}, also the domination by a weighted sum $W:\mathbb{R}_+^n\to\mathbb{R}_+$, $W(\x)=\sum_{i=1}^{n}w_{i}x_i$, with $\mathbf{w}=(w_1,\ldots,w_n)\in\mathbb{R}_+^n\setminus\{\mathbf{0}\}$, is sufficient to guarantee that
$S_{\left(A,\mathcal{R}_n\right)}=\rea^n_+$ (i.e., $A(\y)\leq W(\y)$ for each $\y\in\rea^n_+$ is considered).

\begin{cor}\label{base2}
Let $A:\rea^n_+\to\rea_+$ be a weighting function on $\mathcal{R}_n$ and let, for a fixed $\varepsilon>0$, 
$
\{
\frac{A(\y)}{k}|\ \y\in\rea^n_+,\ \max{\{y_1,\ldots,y_n\}}\leq k
\}
$ 
be bounded by a fixed constant $c$, independently of $k\in\left]0,\varepsilon \right]$. 
Then $S_{(A,\mathcal{R}_n)}=\rea^n_+$.
\end{cor}

The following example shows that, in general (i.e. when $\mathcal{D}\subsetneq \mathcal{R}_n$), Theorem \ref{suff} is not valid.

\begin{examp}
Consider in $\rea^2_+$ the following decomposition system 

\[
\mathcal{D}=
\left\{
\left(\left(1,1\right),\overbrace{\left(\frac{1}{n},\frac{1}{n}\right),\ldots,\left(\frac{1}{n},\frac{1}{n}\right)}^{n \ times}\right)\right\}_{n\in\nat}
\] 

Now, independently from the weighting function, $I_{(A,\mathcal{D})}(1,1)=\sup\emptyset=0$.
On the other hand, if we choose $A(x,y)=x+\sqrt{y}$, we have
\[
I_{(A,\mathcal{D})}(2,2)=\sup\left\{A(1,1)+\sum_1^n A\left(\frac{1}{n},\frac{1}{n}\right)\right\}_{n\in\nat}=\sup\left\{3+\sqrt{n}\right\}_{n\in\nat}=+\infty.
\]
\end{examp}

\begin{examp}\label{workers1}
Let us reconsider Example \ref{workers}.
To choose the best group of teams, we have to compute efficiency of various complete groups (i.e. where we use all four workers), which can be easily done due to small quantity of data. 

\begin{equation}
\left\{
\begin{array}{l}
2\cdot E(1,0) +E(0,2)=2.0+2.0=4.0\\
2\cdot E(1,0) +2\cdot E(0,1)=2.0+2.2=4.2\\
 E(2,0) +E(0,2)=2.2+2.0=4.2\\
E(2,0)+2\cdot E(0,1) =2.2+2.2=4.4.\\
E(1,2) +E(1,0)=3.0+1.0=4.0\\
 E(2,2) =4.3\\
 E(1,1) +E(1,0)+E(0,1) =2.2+1.0+1.1=4.3\\
2\cdot E(1,1) =4.4\\
 E(2,1)+E(0,1) =3.5+1.1=4.6= I_{(E,\mathcal D_\mathcal T)}(2,2).
\end{array} \right.
\label{eq:E1}
\end{equation}

System \eqref{eq:E1} underlines as the best solution corresponds to $I_{(E,\mathcal D_\mathcal T)}(2,2)$.
This example can be generalized, by thinking that the two agencies $\mathcal A_1$ and $\mathcal A_2$ can provide any number of workers and then the possible teams are identified with elements of $\mathcal T=\nat_0^2\setminus\{(0,0)\}$.
Supposing that we know the efficiency of all possible teams, expressed by the weighting function 
$E:\mathcal T\rightarrow \rea_+$ and supposing that the first agency provides $n_1$ workers and the second agency $n_2$, then the best group of teams correspond to decompositions of $(n_1,n_2)$ allowing the computation of $I_{(E,\mathcal D_\mathcal T)}(n_1,n_2)$.
For $n_1$ and $n_2$ large enough we need the use of linear programming techniques to compute $I_{(E,\mathcal D_\mathcal T)}(n_1,n_2)$, however $I_{(E,\mathcal D_\mathcal T)}(n_1,n_2)$ is the theoretical solution to the problem, in the sense that the sub-decomposition integral definition provides the algorithmic to solve the problem.
\end{examp}

Let us consider Examples \ref{workers} and \ref{workers1}. The optimal solution we found, $I_{(E,\mathcal D_\mathcal T)}(2,2)=4.6$, can be also obtained by using the concave integral \cite{Lehrer} and choosing an ``ad hoc'' measure as described in the following.
We identify the set of the four workers with $N=\{1,2,3,4\}$ where 1 and 2 are the two workers from the first agency and 2,3 those from the second. 
Consider the measure $\nu:2^N\rightarrow \rea_+$ given by $\mu(\emptyset)=0$, $\mu(1)=\mu(2)=E(1,0)$, $\mu(3)=\mu(4)=E(0,1)$,
$\mu(12)=E(2,0)$, $\mu(34)=E(0,2)$, $\mu(13)=\mu(14)=\mu(23)=\mu(24)=E(1,1)$, $\mu(123)=\mu(124)=E(2,1)$, $\mu(134)=\mu(234)=E(1,2)$ and $\mu(1234)=E(2,2)$.
Now the best solution for the problem proposed in Example \ref{workers} is given by 
$\int^{cav}{(1,1,1,1)d\nu}=1\cdot\nu(123)+1\cdot \nu(3)=4.6$.
Also the generalization of the problem discussed at the end of Example \ref{workers1} can be obtained using the concave integral, in the sense that $I_{(E,\mathcal D_\mathcal T)}(n_1,n_2)=\int^{cav}{\y d\nu}$ where 
$N=\{1,2,\ldots,(n_1+n_2)\}$, $\y=(1,1,\ldots,1)\in\nat^{n_1+n_2}$ and $\nu:2^N\rightarrow\rea_+$ is an opportune capacity.
However this is possible only because we have chosen an integer components vector $(n_1,n_2)$ and we have allowed only for decomposition of it in integer components vectors.
Suppose to have two numerical control machines $M_1$ and $M_2$ and they can work alone or together, the first machine depends on a parameter $\alpha_1$ and the second on a parameter $\alpha_2$, with $(\alpha_1,\alpha_2)\le(2\sqrt{2},2)$.
The possible setting of these two machines are identified with $\mathcal T=]0,\alpha_1]\times]0,\alpha_2]$, and we know the efficiency of each combination of these machines given by $E:\mathcal T\rightarrow \rea_+$.
finally the best setting for the couple of machines is obtained by solving $I_{(E,\mathcal D_\mathcal T)}(\alpha_1,\alpha_2)$. Suppose that $I_{(E,\mathcal D_\mathcal T)}(\alpha_1,\alpha_2)=E(2\sqrt{2},\sqrt{2})+E(0,2-\sqrt{2})$.
In this case no measure can be specified in order to solve the problem using the concave integral.

\section{Super-decomposition based integrals}\label{superdec}

We open this section with a realistic example, providing motivations to our approach to super-decomposition integral.

\begin{examp}\label{MC2}

Consider a Fast Food (FF) which, basically, offers three goods (\textit{basic-offers})
\[
g_1=\mbox{hamburger},\quad g_2=\mbox{chips},\quad g_3=\mbox{coke.}
\]

To increase the sales, the FF proposes also discounted \textit{compound-offers}, e.g. to buy conjointly 1 [hamburger + chips] is less expansive than 1 hamburger and 1 chips bought separately.
Let us suppose that the FF set of offers is 
$$\mathcal S=\left\{(1,0,0), (0,1,0),(0,0,1),(1,1,1),(2,0,0),(1,0,1),(0,1,1),(2,1,1)\right\},$$
where $(1,0,0)$, $(0,1,0)$ and $(0,0,1)$ represent, respectively, the basic offers hamburger, chips and coke, while, e.g., $(1,1,0)$ represents the compound offer [hamburger + chips]. 
To attract the consumers, FF propose a price function $P:\mathcal S\rightarrow \rea_+$, which is typically strictly sub-additive, i.e., 
$$
P(x,y,z)<\sum_{i=1}^n P(x_i,y_i,z_i),
$$
for all $(x,y,z),\ (x_i,y_i,z_i)\in \mathcal S$ such that $(x,y,z)=\sum_{i=1}^n (x_i,y_i,z_i)$, $n\ge 2$. For example, $P(1,1,1)<P(1,0,1)+P(0,1,0)<P(1,0,0)+P(0,1,0)+P(0,0,1)$.
Let us suppose that FF prices are 
$$
P(1,0,0)=2.80,\  P(0,1,0)=1.60,\  P(0,0,1)=1.80,\  P(1,1,1)=4.80,
$$$$
P(2,0,0)=P(1,0,1)=P(0,1,1)=3,\ \mbox{and}\  P(2,1,1)=5.50.
$$
Let us suppose also that a group of friends have to buy altogether 50 hamburgers, 30 chips and 60 cokes, and, obviously, they want to pay as little as possible by taking advantage of FF offers.
This is a linear programming problem, which can be formalized as follows ($x_a$ is integer quantity of $(1,0,0)$, $P_a=P(1,0,0)$ and so on) 

\begin{equation}
\left\{
\begin{array}{l}
P_G(50,30,60)=\min\{x_a P_a + x_b P_b + x_c P_c 
+ x_{aa} P_{aa} + x_{ac} P_{ac} + x_{bc} P_{bc}+ \\
+ x_{abc} P_{abc} + x_{aabc} P_{aabc}\}
 \\
x_a+2x_{aa}+ x_{ac}+ x_{abc}+ 2x_{aabc}=50\\
x_b + x_{bc}+ x_{abc} + x_{aabc} =30\\
x_c + x_{ac}+ x_{bc} + x_{abc} + x_{aabc}=60\\
x_a,x_b,\ldots,x_{aabc}\ \mbox{integer}.
\end{array} \right.
\label{eq:FF}
\end{equation}

But consider, for example, the necessity to buy 19 hamburgers, 10 chips and 10 cokes. Since 
$5.5\cdot 10 < 5.5\cdot9 + (2.8+1.6+1.8)$, we understand that to find the optimal solution, in equation \eqref{eq:FF} we must replace equality on constrains with inequality, i.e., 
\begin{equation}
\left\{
\begin{array}{l}
P_G(50,30,60)=\min\{x_a P_a + x_b P_b + x_c P_c 
+ x_{aa} P_{aa} + x_{ac} P_{ac} + x_{bc} P_{bc}+ \\
+ x_{abc} P_{abc} + x_{aabc} P_{aabc}\}
 \\
x_a+2x_{aa}+ x_{ac}+ x_{abc}+ 2x_{aabc}\ge 50\\
x_b + x_{bc}+ x_{abc} + x_{aabc} \ge30\\
x_c + x_{ac}+ x_{bc} + x_{abc} + x_{aabc}\ge60\\
x_a,x_b,\ldots,x_{aabc}\ \mbox{integer}.
\end{array} \right.
\label{eq:FF1}
\end{equation}

We will return on this example later, after introducing sub-decomposition based integrals.
\end{examp}

Sub-decomposition based integrals can be considered as an optimization problem to maximize the possible profit. In a dual way modeling the minimization of the costs, one can introduce super-decomposition based integrals.

However, there is a crucial difference concerning the possible inputs $\x\in\mathbb{R}_+^n$ to be evaluated by a super-decomposition based integral. Indeed, for a fixed decomposition system 
$\mathcal{D},\hat{\mathcal{D}}=\{\sum_{j=1}^{k}\y^j|\ \mathcal{B}=(\y^{j})_{j=1}^{k}\in\mathcal{D}\}$ 
is the set of maximal elements of the set of all elements $\x\in\mathbb{R}_+^n$ covered by some collection $\mathcal{B}$ from $\mathcal{D}$, i.e., a super-decomposition based integral can be defined only on the domains $\bar{\mathcal{D}}\subseteq\mathbb{R}_+^n$ given by $\bar{\mathcal{D}}=\{\x\in\mathbb{R}_+^n|\ \x\leq \sum_{j=1}^{k}\y^{j}$ for some collection $\mathcal{B}\in\mathcal{D}\}=\cup_{\y\in\hat{\mathcal{D}}}[\boldsymbol 0,\y]$. Obviously, if $\tilde{\mathcal{D}}=\mathbb{R}_+^{n}$ then also $\bar{\mathcal{D}}=\mathbb{R}_+^n$.\\
Given a base $\left(A,\mathcal D\right)$ and $\x\in\bar{\mathcal{D}}$, it results that 

$$
0\le\inf{\{\sum_{j=1}^{k}A(\y^{j})|\ \x\leq\sum_{j=1}^{k}\y^{j},\ (\y^{j})_{j=1}^{k}\in\mathcal{D}\}}<\infty.
$$
If there exists $\x\in\bar{\mathcal{D}}$ such that $\inf{\{\sum_{j=1}^{k}A(\y^{j})|\ \x\leq\sum_{j=1}^{k}\y^{j},\ (\y^{j})_{j=1}^{k}\in\mathcal{D}\}}>0$, $\left(A,\mathcal D\right)$ is called a base for sup-integration (shortly,a sup-base). For example, $\left(A,\mathcal \mathcal R_n\right)$ is not a feasible base for sup-integration when considering the product $A=\Pi$ or $A=\min$ [consider the decomposition $\x=(x_1,0,\ldots,0)+\ldots +(0,\ldots,0,x_n)$].

\begin{Def}\label{superDeI}
Let $(A,\mathcal{D})$ be a base for sup-integration on $\rea^n_+$, then the $(A,\mathcal{D})$-based super-decomposition integral is the functional $I^{(A,\mathcal{D})}:\bar{\mathcal{D}}\to\mathbb{R}_+$ defined by 
\begin{equation}
I^{(A,\mathcal{D})}(\x)=\inf{\{\sum_{j=1}^{k}A(\y^{j})|\ \x\leq\sum_{j=1}^{k}\y^{j},\ (\y^{j})_{j=1}^{k}\in\mathcal{D}\}},
\end{equation}
\end{Def}

Obviously, if $\mathcal{D}=\mathcal{R}_n$ (an then $\bar{\mathcal{D}}=\mathbb{R}_+^n$), then $I^{(A,\mathcal{D})}:\mathbb{R}_+^n\to\mathbb{R}_+$ is an aggregation function.

\begin{rem}\label{sub}
If an aggregation function $B:\rea^n_+\rightarrow \rea_+$ is sub-additive $[A(\x+\y)\le A(\x)+A(\y)]$, and if 
considering the weighting function $A=B|_{\tilde{\mathcal{D}}}$, the couple $\left(A,\mathcal D\right)$ is a sup-base,
then $I^{(A,\mathcal{D})}(\x)\ge B(\x)$ for each $\x\in\bar{\mathcal{D}}$.
Obviously, if $(\x)\in\mathcal{D}$, then $I^{(A,\mathcal{D})}(\x)= B(\x)$.
%and if there exists $(\y^j)_{j=1}^k\in \mathcal{D}$ such that $\sum_{j=1}^k \y^j =\x$, then $I^{(A,\mathcal{D})}(\x)= A(\x)$.
\end{rem}

\begin{examp}
Continuing in Example \ref{MC2}, we can assume 
$$\tilde{\mathcal D}=\mathcal S=\left\{(1,0,0), (0,1,0),(0,0,1),(1,1,1),(2,0,0),(1,0,1),(0,1,1),(2,1,1)\right\},$$
and $\mathcal D_\mathcal S$ is the decomposition system containing all collections building with elements from $\mathcal S$.
It is clear that the solution of problem \eqref{eq:FF1} (the minimal price that the group should pay to satisfy their constrains) is $I^{(P,\mathcal D_\mathcal S)}(50,30,60)$.
Using a linear programming solver it results 
$$I^{(P,\mathcal{S})}(50,30,60)=10\cdot P_{aabc}+30\cdot P_{ac}+20\cdot P_{bc}=205.$$

\end{examp}

%\begin{examp}
%Continuing in Examples \ref{MC} and \ref{MC1}, observe that $\hat{\mathcal{D}}_1=\nat^3$ and $\bar{\mathcal{D}}_1=\rea^3_+$, while $\hat{\mathcal{D}}_2=\{(x_1,x_2,x_3)\wedge(L_1,L_2,L_3)|(x_1,x_2,x_3)\in\nat^3\}$ and $\bar{\mathcal{D}}_2=\{(x_1,x_2,x_3)\wedge(L_1,L_2,L_3)|(x_1,x_2,x_3)\in\rea^3_+\}$. However, in this real example we are interested only on vectors from $\hat{\mathcal{D}}_1$ and $\hat{\mathcal{D}}_2$, corresponding to integer amount of goods.
%Now consider a group of friends that have to buy altogether at least $M_1$ hamburgers, $M_2$ chips and $M_3$ cokes,    
%then $I^{(P,\mathcal{D}_1)}(M_1,M_2,M_3)$ is the minimal price that the group should pay to satisfy their constrains. $I^{(P,\mathcal{D}_1)}(M_1,M_2,M_3)$ can be easily computed once price function $P$ is specified.
%\end{examp}

\begin{examp}\label{suptconorm}
Let us consider the probabilistic sum (this is a weighting function and a t-conorm) $B:[0,1]^2\to\mathbb{R}_+$ given by $B(x,y)=x+y-xy$ and the decomposition system $\mathcal D_{[0,1]^n}=\{(\x^j)_{j=1}^k\in\mathcal R_n\  |\ \x^j\in[0,1]^n\ j=1,\ldots,k\}$.
Then $\tilde{\mathcal D}_{[0,1]^n}=[0,1]^n$, $\bar{\mathcal{D}}_{[0,1]^n}=\rea^n_+$, and $I^{(B,\mathcal{D}_{[0,1]^n})}:\mathbb{R}_+^2\to\mathbb{R}_+$ is given by
$$
I^{(B,\mathcal{D}_{[0,1]^n})}(x,y)=
\left\{
\begin{array}{ll}
(k+1)(x+y-k)-xy &
\begin{array}{l}
\textrm{if $(x,y)\in[k,k+1]^2$} \\
\textrm{for some $k\in\mathbb{N}$,}
\end{array}\\
\\
\max{(x,y)} &
\begin{array}{l}
\textrm{otherwise.}
\end{array}
\end{array} \right.
$$
Observe that $I^{(B,\mathcal{D}_{[0,1]^n})}$ can be seen as a pseudo-addition on $[0,\infty]$ (when extended by monotonicity also for infinite inputs), \cite{Sugeno1987197}, \cite{KlemMesPap2000}, $I^{(B,\mathcal{D})}=(<k,k+1,B>|\ k\in\mathbb{N}_0)$, i.e., it is associative, commutative aggregation function on $\mathbb{R}_+^2$ with neutral element $e=0$.
Let us note that $I^{(B,\mathcal{D}_{[0,1]^n})}(x,y)+I_{(A,\mathcal{D}_{[0,1]^n})}(x,y)=x+y$ for all $x,y\in[0,\infty]$, i.e., our integrals solves Frank's functional equation \cite{Frank1979}, \cite{KlemMesPap2000} on $[0,\infty]$.
\end{examp}

\mbox{}

\section{Relation with some other integrals}\label{relation}

Let $N=\{1,\ldots,n\}$ be a finite set and let $m:2^N\to\mathbb{R}_+$ be a monotone measure.
Even and Lehrer \cite{EvenLehrer2014} consider a decomposition set $\mathcal{H}$ being a non-empty set of finite systems\footnote{Effectively, Even and Lehrer \cite{EvenLehrer2014} speak about sets whereas we speak about systems. Precisely, they define a \textit{collection} $\mathcal D$ to be a set of subsets of $N$, i.e. $\mathcal D\subseteq 2^N$, and then they consider sets of collections. However their approach can be equivalently given using systems and this allow us to demonstrate that our approach is more general.}
of subsets of $N$, that is $\mathcal{H}=\left\{\mathcal C_1,\ldots,\mathcal C_k\right\}$, with $\mathcal C_i=\left(E^i_{j}\right)_{j=1}^{m_i}$ for all $i=1,\ldots,k$, being $E^i_j\subseteq N$ for all $j=1,\ldots,m_i$. 
The $\mathcal{H}$-decomposition integral is given by
\begin{equation}
I_{\mathcal{H},m}(\x)=\sup{\{\sum_{j=1}^{k}a_{j}m(E_j) |\ (E_j)_{j=1}^k\in\mathcal{H},\ a_1,\ldots,a_k\geq 0,\ \sum_{j=1}^{k}a_j\mathbf{1}_{E_{j}}\leq\x\}}.
\end{equation}
It is not difficult to check that then $I_{\mathcal{H},m}=I_{(A_m,\mathcal{D}_\mathcal{H})}$, where the decomposition system $\mathcal{D}_\mathcal{H}$ is defined by $\mathcal{D}_\mathcal{H}=\{(a_j\mathbf{1}_{E_j})_{j=1}^{k}|\ (E_j)_{j=1}^{k}\in\mathcal{H}, a_1,\ldots,a_k\geq 0\},$ and the weighting function  $A_m:\tilde{\mathcal{D}}_\mathcal{H}\to\mathbb{R}_+$ is given by $A_m(c\cdot \mathbf{1}_E)=c\cdot m(E)$. Thus our approach extends the proposal of Even and Lehrer \cite{EvenLehrer2014}. In particular, it holds:
\begin{itemize}
\item if $\mathcal{H}=\{(E)|E\subseteq N\}$, then $I_{(A_m,\mathcal{D}_\mathcal{H})}$ is the Shilkret integral \cite{shilkret1971maxitive};
\item if $\mathcal{H}= \{\left(E_j\right)_{j=1}^k\ |\ \left(E_j\right)_{j=1}^k$ is a chain$\}$, then $I_{(A_m,\mathcal{D_\mathcal{H}})}$ is the Choquet integral \cite{choquet1953theory};
\item if $\mathcal{H}= \{\left(E_j\right)_{j=1}^k\ |\ \{E_1,\ldots,E_k\}$ is a partition of $N \}$, then $I_{(A_m,\mathcal{D_\mathcal{H}})}$ is the PAN integral \cite{Yang1985}; if $m$ is additive, then the classical Lebesque integral is recovered;
\item if $\mathcal{H}=\{\left(E_j\right)_{j=1}^k\ |\ E_j\subseteq N, \ j=1,\ldots,k\}$, $I_{(A_m,\mathcal{D}_\mathcal{H})}$ is the concave integral \cite{Lehrer}.
\end{itemize}

The couple $(A,\nu)$ is defined a fuzzy capacity \cite{Lehrer} if $(1,\ldots,1)\in A\subseteq [0,1]^n$ and 
$\nu:A\rightarrow\rea_+$ is monotonic, continuous, and there is a positive K such that for every $\boldsymbol a=(a_1,\ldots,a_n)\in A$, it holds $\nu(\boldsymbol a)\le K\sum_{i=1}^n a_i$ .
The concave integral of $\x\in\rea^n_+$ with respect to the fuzzy capacity $(A,\nu) $\cite{Lehrer} is 
$$\int^{cav}{\x d(A,\nu)}=\sup\left\{\sum_{i=1}^k{\alpha_i\nu(\boldsymbol a_i)}\ |\ \boldsymbol a_i\in A, \alpha_i\ge 0, i=1,\ldots,n\mbox{ and }\sum_{i=1}^k{\alpha_i\boldsymbol a_i}\le\x\right\}.$$
If we consider $X=\left\{\alpha\cdot\boldsymbol a\ |\ \alpha\ge 0\mbox{ and }\boldsymbol a\in A\right\}$, 
$\mathcal D_X=\{(\x^i)_{i=1}^k\in\mathcal{R}_n\ |\ \x^i\in X\}$ and the weighting function
$A:X\rightarrow \rea_+$ defined by $A(\alpha\cdot\boldsymbol a)=\alpha\cdot \nu(\boldsymbol a)$ then, it results 
$I_{(A,\mathcal D_X)}(\x)=\int^{cav}{\x d(A,\nu)}$ for all $\x\in\rea_+^n$.\\
For several other integrals covered by our approach we recommend \cite{EvenLehrer2014} \cite{MesStup2013}. 

Recently introduced superadditive integral \cite{GMRS} deals with a fixed decomposition system $\mathcal{D}=\mathcal{R}_n$, and then the weighting function $A$ defined on $\tilde{\mathcal{D}}=\mathbb{R}_+^n$ is just an aggregation function. The superadditive integral $A^*:\mathbb{R}_+^n\to\mathbb{R}_+$ is given by $A^*(\x)=\sup{(\{ \sum_{j=1}^{k}A(\y^j)|\ \sum_{j=1}^{k}\y^j\leq\x \})}$.
Obviously, $A^*=I_{(A,\mathcal{R}_n)}$.

\noindent In the framework of super-decomposition based integrals, we recall that, for a monotone measure $m$:
\begin{itemize}
\item if $\mathcal{H}= \{\left(E_j\right)_{j=1}^k\ |\ \left(E_j\right)_{j=1}^k$ is a chain$\}$, then $I^{(A_m,\mathcal{D}_\mathcal{H})}$ is the Choquet integral;
\item if $\mathcal{H}=\{\left(E_j\right)_{j=1}^k\ |\ E_j\subseteq N, \ j=1,\ldots,k\}\setminus\{(\emptyset)\}$, then $I^{(A_m,\mathcal{D}_\mathcal{H})}$ is the convex integral recently introduced in \cite{MesiarLiPap}.
\end{itemize}

Also the subadditive integral $A_*:\mathbb{R}_+^n\to\mathbb{R}_+$ introduced in \cite{GMRS} can be seen as super-decomposition based integral, $A_*=I^{(A,\mathcal{R}_n)}$.

\subsection{The Choquet integral with respect to a level dependent capacity}

An example of an integral which cannot be considered a sub-decomposition based intgeral is the Choquet integral with respect to a level dependent capacity \cite{greco2011choquet}.
Given a set of criteria $N=\{1,\ldots,n\}$, a level dependent capacity is an index set $(\nu_t)_{t\in\rea_+}$ such that for all $t\in\rea_+$, $\nu_t:2^N\rightarrow [0,1]$ is a capacity.
The Choquet integral of $\x=(x_1,\ldots,x_n)\in\rea^n_+$ with respect to the level dependent capacity $(\nu_t)_{t\in\rea_+}$ is given by 
$Ch_l(\x,\nu_t)=\int_0^\infty{\nu_t(\{i\in N\ |\ x_i\ge t\})dt}$. 
In this case the integral brings too much information to be modeled via a decomposition of the integrated function, $\x=\y_1+\ldots +\y_k$, and weights assigned to addends $w(\y_1)$,...,$w(\y_n)$.
Consider the following example. Given $N=\{1,2,3\}$, and $\x=(3,2,5)$ it results
$Ch_l(\x,\nu_t)=\int_0^2{\nu_t(\{1,2,3\})dt}+ \int_2^3{\nu_t(\{1,3\})dt}+\int_3^5{\nu_t(\{3\})dt}$. 
This integral decomposition ``suggests'' the vector decomposition
$\x=(3,2,5)=(2-0)(1,1,1)+(3-2)(1,0,1)+(5-3)(0,0,1)=(2-0)\boldsymbol 1_{N}+(3-2)\boldsymbol 1_{\{1,3\}}+(5-3)\boldsymbol 1_{\{3\}}$, however to apply the decomposition approach we should assign weights to terns $(a,b,E)\in \rea_+^2\times 2^N$ with $a\le b$, being these weights $\int_a^b{\nu_t(E)dt}$.

\section{Particular decomposition based integrals}\label{particular}

Inspired by set decomposition systems recalled in Section 4, one can define particular vector decomposition systems. Namely we can consider:
\begin{itemize}
\item for a fixed $k\in\mathbb{N}$, $\mathcal{D}_k=\{(\y^j)_{j=1}^{k}|\ \y^{i}$ and $\y^j$ are comonotone for any $i,j\in\{1,\ldots,k \}\}$. Note that if each $\y^j=a_j\cdot\mathbf{1}_{E_j}$ for $a_j>0$ and $E_j\neq\emptyset$, then $(\y^j)_{j=1}^k\in\mathcal{D}_k$ if and only if $(E_j)_{j=1}^{k}$ is a chain in $N$, compare set decomposition system for the Choquet integral; and we denote $\mathcal{D}_{\infty}=\bigcup_{k=1}^{\infty}\mathcal{D}_k$;
\item for a fixed $k\in\{1,\ldots,n\}$, $\mathcal{D}^{(k)}=\{(\y^j)_{j=1}^k|\ supp\ \y^j\cap supp\ \y^{i}=\emptyset$ whenever $i\neq j\}$; these decomposition systems are related to set decomposition system inducing PAN-integral;
\item for a fixed $k\in\mathbb{N}$, $\mathcal{D}_{(k)}=\{(\y^j)_{j=1}^k\}$; clearly, $\mathcal{D}_{(\infty)}=\bigcup_{k=1}^{\infty}\mathcal{D}_{(k)}=\mathcal{R}_n$, and these decomposition systems are related to the concave (convex) integral.
\end{itemize}
Note that for $k=1$, $\mathcal{D}_1=\mathcal{D}^{(1)}=\mathcal{D}_{(1)}=\{(\y)|\ \y\in\mathbb{R}_+^n\}$, and then $\tilde{\mathcal{D}}_1=\mathbb{R}_+^n$. For any aggregation (weighting) function $A:\mathbb{R}_+^n\to\mathbb{R}_+$ it holds $I_{(A,\mathcal{D}_1)}=I^{(A,\mathcal{D}_1)}=A$. Moreover, $I_{(A,\mathcal{D}_{(\infty)})}=A^*$ and $I^{(A,\mathcal{D}_{(\infty)})}=A_*$, compare \cite{GMRS}.

We turn our attention to the decomposition system $\mathcal{D}_{\infty}$ (recall its relation to the Choquet integral). Due to Schmeidler \cite{schmeidler86},\cite{Schmeidler89}, Choquet integral can be characterized by the comonotone additivity. Recall that two vectors $\x$, $\y\in\mathbb{R}_+^n$ are comonotone whenever $(x_i-x_j)(y_i-y_j)\geq0$ for any $i$, $j\in\{1,\ldots,n\}$. The mutual comonotonicity of a collection $\mathcal{C}=(\y^j)_{j=1}^k\in\mathcal{D}_{\infty}$ means that there is a common chain $(E_r)_{r=1}^n$ in $N$ such that each $\y^j,\ j\in\{1,\ldots,k\}$, can be expressed as a linear combination $\y^j=\sum_{r=1}^n a_{r,j}\cdot \mathbf{1}_{E_r}$, with non-negative constants $a_{r,j}$. Moreover for any set $E\subseteq N$, the minimal values of set $\{y_i^j|\ i\in E\}$, $j=1,\ldots,k$, are attained in a single coordinate $i_E\in E$. This observation has an important consequence formalized in the next Lemma.

\begin{Lemma}\label{lemma}
Let $\x$, $\z\in\mathbb{R}_+^n$ be comonotone and let $\x=\sum_{j=1}^k\y^j,\ \z=\sum_{i=1}^m \mathbf{u}^i$, where $(\y^j)_{j=1}^k$ and $(\mathbf{u}^i)_{i=1}^m$ are comonotone systems. Then also $((\y^j)_{j=1}^k,\ (\mathbf{u}^{i})_{i=1}^m)$ is a comonotone system.
\end{Lemma}
Based on Lemma \ref{lemma}, we have the next characterization of $I_{(A,\mathcal{D}_{\infty})}$.

\begin{Theo}
Let $A:\mathbb{R}_+^n\to\mathbb{R}_+$ be an aggregation function such that $S_{\left(A,\mathcal{D}_{\infty}\right)}=\mathcal{R}_n$. Then $I_{(A,\mathcal{D}_{\infty})}$ is the smallest comonotone superadditive aggregation function dominating $A$, and for each $\x\in\mathbb{R}_+^n$, $I_{(A,\mathcal{D}_{\infty})}(\x)=\min{\left\{C(\x)|\ C\geq A\right.}$, $C$ is a comonotone superadditive aggregation function\}.
\end{Theo}
\begin{proof}
We only prove the comonotone superadditivity of $I_{(A,\mathcal{D}_{\infty})}$, while the rest of proof can be done similarly as in \cite{GMRS} (Proposition 2). Fix a comonotone couple $\x$, $\z\in\mathbb{R}_+^n$. Based on Lemma \ref{lemma}, (it implies the first inequality)

$$I_{(A,\mathcal{D}_{\infty})}(\x+\z)=\sup{\{\sum_{r=1}^p A(\mathbf{v}^r)|\ \sum_{r=1}^p \mathbf{v}^r=\x+\z,\ (\mathbf{v}^r)_{r=1}^p\in\mathcal{D}_{\infty} \}}\geq$$
$$\geq \sup{\{\sum_{j=1}^k A(\y^j)+\sum_{i=1}^m A(\mathbf{u}^i)|\ \sum_{j=1}^k \y^j=\x,\ \sum_{i=1}^m \mathbf{u}^i=\z,(\y^j)_{j=1}^k,\ (\mathbf{u}^{i})_{i=1}^m\in\mathcal{D}_{\infty}\}}\geq$$
 $$\geq\sup{\{\sum_{j=1}^k A(\y^j)|\ \sum_{j=1}^k \y^j=\x,\ (\y^j)_{j=1}^k\in\mathcal{D}_{\infty}\}}+$$
$$+\sup{\{\sum_{i=1}^m A(\mathbf{u}^i)|\ \sum_{i=1}^m \mathbf{u}^i=\z,\ (\mathbf{u}^i)_{i=1}^m\in\mathcal{D}_{\infty}\}}= $$
 $$=I_{(A,\mathcal{D}_{\infty})}(\x)+I_{(A,\mathcal{D}_{\infty})}(\z).$$
\begin{flushleft}
$\square$

\end{flushleft}
\end{proof}

\begin{examp}
Define $A:\mathbb{R}_+^2\to\mathbb{R}_+$ by $A(x,y)=\max{(\ln (1+x),\ \ln (1+y))}$. Then $A^*(x,y)=I_{(A,\mathcal{D}_{\infty})},\ (x,y)=x+y$, but $A_{(A,\mathcal{D}_{\infty})}(x,y)=\max{\{x,y\}}$. Observe that $\max{}$ is not superadditive but it is comonotone superadditive.
\end{examp}
A similar result can be shown where considering $\mathcal{D}^{(\infty)}$ decomposition system. We omit its proof due to its simplicity.

\begin{Theo}
Let $A:\mathbb{R}_+^n\to\mathbb{R}_+$ be an aggregation function such that $S_{\left(A,\mathcal{D}^{(\infty)}\right)}=\mathcal{R}_n$.
Then $I_{(A,\mathcal{D}^{(\infty)})}$ is the smallest aggregation function which is superadditive for vectors with disjoint supports, i.e., $I_{(A,\mathcal{D}^{(\infty)})}(\x+\z)\geq I_{(A,\mathcal{D}^{(\infty)})}(\x)+I_{(A,\mathcal{D}^{(\infty)})}(\z)$ whenever $\x\wedge\z=\mathbf{0}$.
\end{Theo}

Similar chains can be shown for the super-decomposition based integrals.

\section{Conclusions}\label{concluding}
In this paper we have studied decomposition approaches to integration generalizing previous works (see \cite{MesStup2013}, \cite{EvenLehrer2014} and \cite{MesiarLiPap}). 
Our general approach to integration is based on three steps: (a) sub/super sum decomposition of integrated functions;
(b) weighting of the addend functions used in decompositions; (c) sum aggregation of these weighted addend functions and choice of extremal elements ($\sup/\inf$) to define the integral.
The final integral depends (other that on the choice of sub/super-decomposition) on the set of allowable functions used to decompose the integrated function in step (a), and on the weighting function used to weigh addend functions 
in step (b).
Note that this approach can be further generalized by replacing standard addition in step (c) with a pseudo-addition.
For example taking any decomposition system $\mathcal D$ such that 
$\tilde{\mathcal{D}}= \{c\cdot 1_E\ |\  c\in [0,\infty], E \subseteq  N\}$, and putting as pseudo-addition $\max$, and as weighting function $A(c\cdot 1_E) = c\cdot m(E)$ being $m:2^N\rightarrow \rea_+$ a measure, the resulting integral is the Shilkret integral; if $A(c\cdot 1_E) = \min(c,m(E))$, Sugeno integral is obtained. 
$\tilde{\mathcal{D}}$ can be finite, consider Ali Baba in the cave with precious things from Gold, only their weight matters, since his donkey can take only $\x$ kg. Ali Baba can take any good he wants, but only one. 
In this case we have $\tilde{\mathcal{D}}=\{g_1,\ldots,g_k\}$, $g_i$ are all possible precious goods in the cave, characterized by their weight $g_i$ and, then, the weighting function is $A(d) = d$, and thus $I^{\max}_{(A,\mathcal D)}(\x) = \max\{g_i\ |\ g_i \le \x\}$. 
Note that if Ali Baba has no limitation in the number of goods but only in the weight $\x$, we have to use our approach based on addition, and then $I_{(A,\mathcal D)}(\x) = \max\{\sum_{i\in I}{g_i}\ | \sum_{i\in I}{g_i} \le \x\}$ and, then, surely $I_{(A,\mathcal D)}(\x)\ge I^{\max}_{(A,\mathcal D)}(\x)$.
This last example recalls a very famous example in literature, the so called \textit{knapsack problem} \cite{martello1990knapsack}. The knapsack problem or rucksack problem is a problem in combinatorial optimization, where, given a set of items, each with a mass and a value, we have to determine the number of each item to include in a collection so that the total weight is not greater than a given limit and the total value is as large as possible. 
The knapsack problem has been studied for more than a century (for example in combinatorics or in the field of resource allocation), and it is straightforward that it can be faced by using our sub-decomposition based integration.

Let us note that in the last step of our construct method for decomposition integrals, we choose the extremal elements of the set of weighted addend functions, that is $I_{(A,\mathcal D)}(\x)=\sup\{\ldots\}$ and $I^{(A,\mathcal D)}(\x)=\inf\{\ldots\}$ and this to link our integrals to optimization problems that usually arise in economics.
Once again, a further generalization is to define the decomposition integral not as the extremal element of the set of all weighted sums of integrated function decompositions, but as a representative element of this set, and, finally, we could consider as integral the whole set, following an approach a l\'a Aumann \cite{aumann1965integrals}. 

\section*{Acknowledgment}
The work on this contribution was partially supported by VEGA 1/0425/15.

\end{document}